\documentclass[12pt]{amsart}

\usepackage{amssymb,amsmath,mathtools}
\usepackage{bbm,enumerate,multirow}
\usepackage{a4wide}
\usepackage{graphicx}

\newtheorem {theorem} {Theorem} 
\newtheorem{prop}[theorem] {Proposition}
\newtheorem{lemma}[theorem] {Lemma}
\newtheorem{coro}[theorem] {Corollary}

\theoremstyle{definition}
\newtheorem{defin}[theorem] {Definition}
\newtheorem{remark}[theorem] {Remark}

\newcommand{\ts}{\hspace{0.5pt}}

\newcommand{\CC}{\mathbb{C}\ts}
\newcommand{\RR}{\mathbb{R}\ts}
\newcommand{\QQ}{\mathbb{Q}\ts}
\newcommand{\ZZ}{\mathbb{Z}}
\newcommand{\TT}{\mathbb{T}}
\newcommand{\NN}{\mathbb{N}}

\newcommand{\abs}[1]{\left|#1\right|}
\newcommand{\one}{\mathbbm{1}}
\newcommand{\id}{\mathbbm{1}}
\newcommand{\exSet}{E}

\DeclareMathOperator{\mgcd}{mgcd}
\DeclareMathOperator{\ord}{ord}

\DeclareMathOperator{\tr}{tr}

\DeclareMathOperator{\GL}{GL}

\DeclareMathOperator{\Per}{Per}
\DeclareMathOperator{\PerP}{PerP}
\DeclareMathOperator{\MPer}{MPer}

\DeclareMathOperator{\Mat}{Mat}
\DeclareMathOperator{\End}{End}

\begin{document}
\title[Period sets of linear toral endomorphisms on $\TT^2$]
{Period sets of linear toral endomorphisms on $\TT^2$}

\author[J. Llibre]{Jaume Llibre $^{1}$}
\address{$^{1,2}$Departament de Matem\`{a}tiques,
Universitat Aut\`{o}noma de Barcelona, 08193 Bellaterra, Barcelona,
Catalonia, Spain} \email{jllibre@mat.uab.cat}

\author[N. Neum\"arker]{Natascha Neum\"arker $^{1,2}$}
\address{$^{2}$Matem\'atick\'y \'ustav Slezsk\'e Univerzity v Opav\v{e},
74601 Opava, \v{C}esk\'a republika}
\email{natascha.neumarker@math.slu.cz}



\begin{abstract}
The period set of a dynamical system is defined as the subset of all
integers $n$ such that the system has a periodic orbit of length
$n$. Based on known results on the intersection of period sets of
torus maps within a homotopy class, we give a complete
classification of the period sets of (not necessarily invertible)
toral endomorphisms on the $2$--dimensional torus $\TT^2$.
\end{abstract}

\maketitle

\section{Introduction and statement of the main results}

The sets of periods that are present in a dynamical system is one of
the key quantities which characterises the system. For continuous
maps on the interval, the celebrated theorem by Sharkovsky provides
a complete characterisation of period sets in terms of the
Sharkovsky ordering. A generalisation of Sharkovsky's Theorem to the
circle was obtained by the authors  Block, Coppel, Guckenheimer,
Misiurewicz and Young \cite{B, BC, BGMY, Mi}, for a unified
proof see \cite{ALM}. For other classes of maps where the period sets
have been studied see for instance \cite{ALM2, H, K, Ll}.

\smallskip

Toral or torus endomorphisms are continuous mappings of the torus that preserve
its group structure, 
hence in the additive notation $\TT^m \cong \RR^m/\ZZ^m$, they can be 
represented as $m\times m$ integer matrices, see for instance 
\cite[Chapter~0]{W}.
In the present article, the term `toral endomorphism' is always used in this
sense, that is, for a map on the torus which is induced by the action
of an integer matrix modulo 1.
They serve as a standard example in the theory of discrete dynamical 
systems and ergodic theory, and particularly the case of hyperbolic 
toral automorphisms,
corresponding to integer matrices with determinant $\pm 1$ and no eigenvalues
on the unit circle, has been studied extensively because of its
interesting dynamical properties, compare \cite{W,KH}.
In \cite{TP}, the period sets of toral automorphisms on $\TT^2$ 
were investigated.

\smallskip

Minimal period sets are the intersection of all period sets arising
from maps in the same homotopy class; the origins of their study go
back to Alsed\'a, Baldwin, Llibre, Swanson and Szlenk \cite{ABLSS},
see also \cite{JM, JL}. Associated with each homotopy class, one has
a unique integer matrix $A$, which defines the action on the first
homology group, and this integer matrix, in turn, defines an
endomorphism of the torus, $f^{}_A$, hence it is itself a member of the
homotopy class. An important result in arbitrary dimension is that
the minimal period set of a homology class ``essentially'' coincides
with the period set of the associated toral endomorphism, apart from
possibly those periods that may arise from roots of unity among the
eigenvalues of this matrix, see \cite{JL}. However, in many of the
cases in which $\pm 1$ is among the eigenvalues, one finds that the
minimal period set and the period set of the endomorphism totally
differ and it can be concluded that, in these cases, the
endomorphism is not a good model for the dynamics of its homotopy
class with respect to periodic orbits.

\smallskip

For circle maps, the situation is comparatively simple and is
displayed in Table~1. The period sets only depend on the
degree $d$ of the map, listed in the first column; the second and
the third column refer to the minimal period set and the period set
of the corresponding endomorphism, respectively; the last column
answers the question whether the endomorphism defining the class
under consideration is invertible, hence an automorphism.
Here, the linear endomorphism $f_A$ is given by $f^{}_{(d)}(x) = d\cdot x$.

\bigskip
\bigskip

\begin{center}
\begin{tabular}{|c|c|c|c|}
 d    &    $\MPer(f_A)$     & $\Per(f_A)$    &  $\in\mathrm{Aut(\TT)}$?
 \\\hline
 1    & $\emptyset$       & $\{1\}$    &   yes \\ \hline
 0    & $\{1\}$       & $\{1\}$        &  no \\ \hline
 -1   & $\{1\}$       & $\{1,2\}$      &  yes \\ \hline
 -2   & $\NN\setminus\{2\}$  & $\NN\setminus\{2\}$    &  no \\ \hline
 $d\in\ZZ\setminus \{-2,-1,0,1\}$   & $\NN$      & $\NN$      & no \\\hline
\end{tabular}
\end{center}
\begin{center}
{\bf Table 1. Period sets for circle maps.}
\end{center}

\bigskip

In this article, we compare the minimal period sets of maps on
$\TT^2$ with the period set of the associated toral endomorphism,
aiming for an analogue of the above table for dimension $2$.

\smallskip

To obtain a complete classification, a variety of partly very
different techniques is employed. The minimal period sets were
derived in \cite{ABLSS} by means of estimating Nielsen numbers, see
\cite{H} for background reading on this approach. The study of
period sets in the special case of toral automorphisms in \cite{TP}
is based on an extensive distinction of cases; we note that much of
the reasoning in \cite{TP} makes use of the assumption of a
determinant $\pm 1$ and hence does not directly generalise to
arbitrary $2\times 2$ matrices. 
We complete the classification of period sets by making use of 
results on \textit{local conjugacy}, that is conjugacy modulo $n$,
$n\in\NN$, which corresponds to the action of the matrix on the
invariant subsets of points with rational coordinates.

\smallskip

The outline of this article is as follows. In Section
\ref{S:knownStuff} we compile the theory of minimal period sets and
periods of toral endomorphisms as far as needed for our purpose,
derive the period sets $\Per(f_A)$ from the minimal period sets
$\MPer(f_A)$ wherever
possible and identify the cases that require individual treatment.
In Section \ref{S:completion}, we complete the classification of
period sets on the two-dimensional torus by considering normal forms
for matrices with an eigenvalue $\pm 1$.
At the end of this introductory section, we summarise the results of our
later analysis in form of a comprehensive table.

\smallskip

The following Table 2 constitutes the $2$--dimensional analogue of Table 1 for
circle maps. Starting from some $A\in\Mat(2,\ZZ)$, the columns (from
left to right order) refer to the eigenvalues of $A$, the pair of
the trace and the determinant of $A$, i.e. $(t,d)= (\tr(A),\det(A))$,
the minimal polynomial $\mu_A^{}$, the minimal period
set $\MPer(f_A)$ of the homotopy class of $f_A$, the period set
$\Per(f_A)$, and finally an answer to the question whether 
$f_A\in\End(\TT^2)$ is invertible, hence an element of
$\textrm{Aut}(\TT^2)\subset\End(\TT^2)$.
The symbol $\chi^{}_A$ denotes the characteristic polynomial of
the matrix $A$; the symbol $E$ refers to the set of exceptional integer
values $E= \{-2,-1,0,1\}$.

\bigskip
\bigskip

\newcounter{rownmbr}\addtocounter{rownmbr}{1}
\begin{center}
\begin{small}
\begin{tabular}{|r|l|c|c|c|c|c|}
\label{sum2D}
& eigenvalues    & $(t,d)$ & $\mu^{}_A$ & $\mathrm{MPer}(f_A)$ & $\mathrm{Per}(f_A)$  & 
$\in\mathrm{Aut}(\TT^2)$?\\ \hline\hline
\arabic{rownmbr}. & $1$  & $(2,1)$ & $x-1$   & $\emptyset$  & $\{1\}$ &  y\\ \hline
\addtocounter{rownmbr}{1}
\arabic{rownmbr}.  & $1$ & $(2,1)$ &  $\chi^{}_A$     & $\emptyset$  & $\NN$ & y\\\hline
\addtocounter{rownmbr}{1}
\arabic{rownmbr}.  & $-1$ & $(-2,1)$ & $x+1$   & $\{1\}$     & $\{1,2\}$ & y\\\hline
\addtocounter{rownmbr}{1}
\arabic{rownmbr}.  & $-1$ & $(-2,1)$ & $\chi^{}_A$        & $\{1\}$     & $2\NN\cup\{1\}$ & y\\\hline
\addtocounter{rownmbr}{1}
\arabic{rownmbr}. & $\pm 1$  & $(0,-1)$  & $\chi^{}_A$    & $\emptyset$     & $\{1,2\}$ &  y\\\hline
\addtocounter{rownmbr}{1}
\arabic{rownmbr}.  & $e^{2\pi i/3},e^{-2\pi i/3}$ & $(-1,1)$  & $\chi^{}_A$  & $\{1\}$ & $\{1,3\}$ & y\\\hline
\addtocounter{rownmbr}{1}
\arabic{rownmbr}. & $\pm i$& $(0,1)$ & $\chi^{}_A$    & $\{1,2\}$     & $\{1,2,4\}$ & y\\ \hline
\addtocounter{rownmbr}{1}
\arabic{rownmbr}. & $e^{\pi i/3},e^{-\pi i/3}$ & $(1,1)$  & $\chi^{}_A$ & $\{1,2,3\}$     & $\{1,2,3,6\}$ & y\\ \hline
\addtocounter{rownmbr}{1}
\arabic{rownmbr}.  & $0$ & $(0,0)$  & $x^2$ or $x$ & \multicolumn{2}{|c|}{$\{1\}$} & n\\\hline
\addtocounter{rownmbr}{1}
\arabic{rownmbr}.  & $0,1$ & $(1,0)$  & $x^2-x$ & $\emptyset$ & {$\{1\}$} & n\\\hline
\addtocounter{rownmbr}{1}
\arabic{rownmbr}.  & $0,-1$ &  $(-1,0)$ & $\chi^{}_A$ & $\{1\}$     & $\{1,2\}$ & n\\\hline
\addtocounter{rownmbr}{1}
\arabic{rownmbr}. & $\notin\RR$ & $(-2,2)$ & $\chi^{}_A$  & \multicolumn{2}{|c|}{$\NN\setminus\{2,3\}$} & n\\\hline 
\addtocounter{rownmbr}{1}
\arabic{rownmbr}. & $\notin\RR$ & $(-1,2)$  &  $\chi^{}_A$& \multicolumn{2}{|c|}{$\NN\setminus\{3\}$} & n\\\hline %
\addtocounter{rownmbr}{1}
\arabic{rownmbr}.  & $\notin\RR$ & $(0,2)$  & $\chi^{}_A$ & \multicolumn{2}{|c|}{$\NN\setminus\{4\}$} & n\\\hline %
\addtocounter{rownmbr}{1}
%
\multirow{2}{*}{\arabic{rownmbr}.} & \multirow{2}{*}{$\notin\RR$} & none of &  \multirow{2}{*}{$\chi^{}_A$} &  \multicolumn{2}{|c|}{\multirow{2}{*}{$\NN$}} & \multirow{2}{*}{y/n} \\
 & & the above    & &\multicolumn{2}{|c|}{}  & \\\hline
\addtocounter{rownmbr}{1}
\multirow{2}{*}{\arabic{rownmbr}.} & real, & \multirow{2}{*}{$t+d\notin \{0,-2\}$} & $\chi^{}_A$ or $x-a$, 
& \multicolumn{2}{|c|}{\multirow{2}{*}{$\NN$}} & \multirow{2}{*}{y/n}\\
 & both $\not=\pm 1$ & & $a\in\ZZ\setminus E$& \multicolumn{2}{|c|}{} & \\\hline
\addtocounter{rownmbr}{1}
\multirow{2}{*}{\arabic{rownmbr}.}  & \multirow{2}{*}{real} & $t+d\in\{0,-2\}$, & \multirow{2}{*}{$\chi_A^{}$ or $x+2$} & \multicolumn{2}{|c|}{\multirow{2}{*}{$\NN\setminus \{2\}$}} & \multirow{2}{*}{y/n} \\
    &     & $(t,d)\not= (0,0)$& &  \multicolumn{2}{|c|}{} &  \\\hline
\addtocounter{rownmbr}{1}
\multirow{2}{*}{\arabic{rownmbr}.} & \multirow{2}{*}{$-1,-d$} & $t+d=-1$, &  \multirow{2}{*}{$\chi^{}_A$}& \multirow{2}{*}{$2\NN-1$} & \multirow{2}{*}{$\NN$} & \multirow{2}{*}{n} \\
& & $d\in\ZZ\setminus\{-1,0,1\}$ & &   & &  \\ \hline
\addtocounter{rownmbr}{1}
\arabic{rownmbr}. & $1,-2$ & $(-1,-2)$ &  $\chi^{}_A$ &$\emptyset$  & $\NN\setminus\{2\}$ & n \\\hline
\addtocounter{rownmbr}{1}
\multirow{2}{*}{\arabic{rownmbr}.}  & \multirow{2}{*}{$1,d$} & $t-d=1$, & \multirow{2}{*}{$\chi^{}_A$} & \multirow{2}{*}{$\emptyset$}& \multirow{2}{*}{$\NN$} & \multirow{2}{*}{n}\\
 & &  $(t,d)\not=(-1,-2)$ & 
&   &  &  \\\hline 
\end{tabular}
\\
\end{small}
\end{center}

\bigskip

\begin{center}
{\bf Table 2. Period sets for torus maps.}
\end{center}

\section{Definitions and overview of results}\label{S:knownStuff}

The set of periods of a map $f:\TT^m\rightarrow\TT^m$ will be
denoted by
\[
\Per(f) \coloneqq \{n\in\NN \mid f\ \text{has an orbit of length}\
n\}.
\]
The \emph{minimal} set of periods is defined as
the intersection of all period sets in a given homotopy class,
\begin{equation}\label{E:IntersectionMPer}
\MPer(f_A) \coloneqq \bigcap_{g\simeq f} \Per(g),
\end{equation}
where $g\simeq f$ means that $g$ is homotopic to $f$.

Associated with a torus map $f:\TT^m \rightarrow \TT^m$, one has the
first induced homology map $f_{*1}:H_1(\TT^m,\ZZ) \rightarrow
H_1(\TT^m,\ZZ)$, which corresponds to some $m\times m$ integer
matrix $A$. In the following, we will denote by $\Mat(m,R)$ the ring
of $m\times m$ matrices with entries from the ring $R$. By
$\Mat(m,R)^{\times}$ we will refer to the invertible matrices over $R$
or, equivalently, to those elements of $\Mat(m,R)$ whose determinant 
is a unit in $R$. 
The linear map defined by the matrix $A : \RR^m\rightarrow
\RR^m$, covers a unique torus endomorphism
$f_A:\TT^m\rightarrow\TT^m$ whose action is given by matrix
multiplication modulo $1$, that is, two points $x,y\in\RR^m$ cover
the same element of $\TT^m$ if and only if $x-y\in\ZZ^m$. By
$\End(\TT^m)$ we denote the set of toral endomorphisms, and by
$f_A\in\End(\TT^m)$ we refer to the map induced by
$A\in\Mat(m,\ZZ)$. As $f_A \simeq f$, the intersection on the
right-hand side of Equation \eqref{E:IntersectionMPer} comprises the
period set of $f_A$, so one always has $\MPer(f_A)\subset
\Per(f_A)$. In fact, these two sets are typically more closely
related; in \cite[Proposition~3.4]{ABLSS} the following general
result is proved
\begin{equation}\label{E:MPerPerOfEndo}
   \MPer(f) = \Per(f_A)\setminus\{k\in\NN : 1\
   \text{is an eigenvalue of $A^n$}\}.
\end{equation}

In other words, the period set of the endomorphism $f_A$ and the
minimal period set $\MPer(f)$ coincide if and only if $A$ does not
have any eigenvalues that are roots of unity. In this case, the
number of fixed points of $A^n$ can be calculated in terms of the
\emph{Nielsen numbers}: $f_A^n$ has $N(f_A^n)$ isolated fixed
points, where
\[
   N(f_A^n) = \abs{\det(\id - A^n)},
\]
and $\id$ denotes the $m\times m$ identity matrix.
For $m=2$, this simplifies to 
\[
   N(f_A^n) = \abs{1+ \lambda_1^n\lambda_2^n -(\lambda_1^n+\lambda_2^n)},
\]
where $\lambda_1,\lambda_2$ are the eigenvalues of $A$.

\smallskip

When the matrix $A$ has an eigenvalue that is an $n$-th root of
unity, the Nielsen numbers vanish for all multiples of $n$. In this
case, $n$ is not in the minimal period set; however, the endomorphism
$f_A$ admits periodic points of period $n$, which form
subtori of $\TT^m$. A treatment of this case can be found in the
appendix of \cite{BHP}.

\smallskip

The rational points of $\TT^m$ form an invariant subset of $\TT^m$
and can be written as the countable union of the finite sets of
$n$-division points, $n\in\NN$,
\[
[0,1)^m \cap\QQ^m = \bigcup_{n\in\NN} \left\{\left(\frac{k_1}{n},
\ldots,\frac{k_m}{n}\right) \mid 0\leq k_1,\ldots, k_m < n\right\} =
\bigcup_{n\in\NN} L_n.
\]
Each rational lattice $L_n$ in the union on the right-hand side is
invariant under the action of the integer matrix $A$, hence $L_n$ is
partitioned into finite orbits by $f_A$. Denoting by $\PerP(f)$ the
periodic \emph{points} of a map $f$, we can state that, if
$\gcd(n,\det(A))=1$, $L_n \subset \PerP(f_A)$, hence
$$
\bigcup_{n\in\NN \, : \, \gcd(n,\det(A))=1} L_n \subset \PerP(f_A),
\quad \mbox{and} \quad  \bigcup_{n\in\NN} L_n \subset \PerP{(f_A)}
$$
if $\det(A)=\pm1$, that is, if $A$ is an automorphism. In the
absence of roots of unity among the eigenvalues, all periodic points
are isolated and live on the rational lattices, whence the converse
$\PerP(f_A) \subset \bigcup_{n\in\NN} L_n$ follows, and one has strict
equality if
$A$ is invertible, as all lattice points are then consumed in periodic
orbits. 
There is an abundance of literature dealing with
the periods of toral endomorphisms (mostly restricted to the
invertible case, i.e. automorphisms) on the rational lattices $L_n$
depending on $n$; for an overview, see, for instance, the articles
\cite{PV, BF, BRW, BNR} and references therein.
The rational lattice $L_n$ can be identified with the free module
over the finite ring $\ZZ/n\ZZ$, so studying the action of an
integer matrix on $L_n$ amounts to studying the matrix action modulo
$n$, and the question whether two integer matrices share the same
orbit structure on $L_n$ is closely related to the question of
conjugacy over the residue class ring $\ZZ/n\ZZ$. Two matrices that
are conjugate over $\ZZ/n\ZZ$ necessarily exhibit the same orbit
structure on $L_n$.

\smallskip

On the two-dimensional torus, one can formulate a necessary and
sufficient criterion for matrices to be locally conjugate on a
lattice $L_n$ for some or all $n\in\NN$ in terms of trace,
determinant and a third invariant, which we now define.

\begin{defin}
The matrix gcd (short mgcd) of a matrix $M= \displaystyle
\left(\begin{matrix} \alpha & \beta\\\gamma & \delta\end{matrix}\right)$
$\in\Mat(2,\ZZ)$ is defined as $\mgcd(M) \coloneqq \gcd(\beta,\gamma,\delta-\alpha)$.
\end{defin}

We note that the square of $\mgcd(M)$ always divides the discriminant
$\Delta (M) = \tr(M)^2-4\det(M)$,
\begin{equation}\label{E:mgcdDivDiscr}
\Delta(M) = (\alpha+\delta)^2 - 4(\alpha \delta -\beta \gamma) 
= (\alpha + \delta)^2 - 4 \beta \gamma.
\end{equation}
The following theorem asserts trace, determinant and mgcd to be a
complete set of invariants with respect to local conjugacy.

\begin{theorem}[{\cite[Theorem~2]{BRW}}]\label{T:mgcd}
For two integer matrices $M,M^{\prime}\in\Mat(2,\ZZ)$, the following
statements are equivalent.
\begin{itemize}
\item[(a)] The reductions mod $n$ of $M$ and $M^{\prime}$ are
$\Mat(2,\ZZ/n\ZZ)^{\times}$ conjugate.

\item[(b)] $M$ and $M^{\prime}$ share the same trace, determinant and
mgcd.
\end{itemize}
\end{theorem}

We will make use of this theorem in Section~\ref{S:completion} in
order to determine the period sets of certain toral endomorphisms
with integer eigenvalues. Now we first treat the cases of complex
and irrational real eigenvalues, for which the period sets of toral
endomorphisms can be gained from the minimal period sets of their
homotopy classes.

\subsection{Complex eigenvalues and finite orders}

If $\xi\in\CC\setminus\RR$ is the eigenvalue of a matrix
$A\in\Mat(2,\ZZ)$, its complex conjugate $\bar{\xi}$ is the second
eigenvalue of $A$ and the determinant $d = \xi\bar{\xi} =
\abs{\xi}^2$ is consequently a positive integer.

\smallskip

If $d=1$, both roots lie on the unit circle, that is, they have
modulus $1$ and they are actually roots of unity. (This is a special
case of a more general fact: if all algebraic conjugates of an
algebraic integer $\alpha$ have modulus $1$, $\alpha$ is a root of
unity.) But if all eigenvalues of the matrix $A$ are distinct roots
of unity, $A$ has finite order, that is, there is some integer $n$
such that $A^n = \id$.

\smallskip

It is well-known that for each $n$ the finite orders of subgroups of
$\GL(n,\ZZ)$ can only assume certain values, determined by all factors
of the cyclotomic polynomials that are elements of $\ZZ[x]$ and have 
degree $n$, hence are the characteristic polynomial of a matrix
in $\GL(n,\ZZ)$.
Furthermore, for $2n$ and $2n+1$, the lists of possible
orders coincide, more specifically, $\GL(n,\ZZ)$ for $n=2$ and $3$
can have subgroups of the finite orders $1,2,3,4,6$, see \cite{KP},
for example.

\smallskip

For $2\times 2$ integer matrices, these cases correspond to the
characteristic polynomials $x^2+x+1$, $x^2+1$, $x^2-x+1$ whose roots
are $e^{2 i \pi/3},e^{-2 i \pi/3}$, $\pm i$ and $e^{i\pi/3},e^{-i\pi/3}$,
respectively. In \cite{TP} the periods of the corresponding
automorphisms are derived from the characteristic equations; using
the knowledge about the minimal period sets, one can treat these
three cases at once. 
\textit{They are represented in rows 6--8 of Table~2}.

\begin{lemma}
   Assume $A\in\Mat(2,\ZZ)$ is of finite order $k=\ord(A)$.
   Then $k\in\Per(f_A)$.
   Consequently, for $A\in\Mat(2,\ZZ)$ whose eigenvalues are complex roots
   of unity, one finds $\Per(f_A) = \MPer(f_A)\cup\{\ord(A)\}$.
\end{lemma}
\begin{proof}
The fact that $A$ has order $k$ means $A^k - \id$ is the zero matrix
but for every $d<k$, $A^d-\id$ is not zero. But that means, there is
some $n\in\NN$ such that $A^d - \id \not\equiv 0 \mod n$ for every
$d<k$. Consequently, there is some $x = (\alpha,
\beta)\in(\ZZ/n\ZZ)^2$ such that $(A^d-\id)x \not= 0$ over
$\ZZ/n\ZZ$ for every $d<k$. Hence, the point $x$ has period $k$. The
last statement follows because the period set of a matrix of order
$k$ can contain divisors of $k$ only, and in the cases under
consideration, all divisors $d$ of $k$ with $d<k$ are already 
contained in the set 
$\MPer(f_A)\subset\Per(f_A)\subset\{d: d|k\}$.
\end{proof}

Note that the matrix order is absent from the set of minimal periods
of the homotopy class, as for $A^n=\id$, one has $N(f^{j\cdot
n}_A)=0$ for all $j\in\NN$. The cases of integer roots of unity will
be dealt with in Section \ref{S:completion}.

\smallskip

If the eigenvalues in $\CC\setminus\RR$ are not roots of unity, the
determinant has to be greater than $1$. The period sets of the
remaining cases with non-real eigenvalues are given by the following
theorem \cite[Thm.~4.24]{ABLSS}.

\begin{theorem}[\cite{ABLSS}]\label{CnoRU}
If $d\geq 2$ and $\lambda,\mu\in\CC\setminus\RR$ then
$\MPer(f_A)=\NN$, unless the pair $(t,d)$ is one of the following
cases:
\begin{itemize}
\item[(a)] If $(t,d)=(-2,2)$, then $\MPer(f_A) =\NN\setminus\{2,3\}$.

\item[(b)] If $(t,d)=(-1,2)$, then $\MPer(f_A) =\NN\setminus\{3\}$.

\item[(c)] If $(t,d)=(0,2)$, then $\MPer(f_A) =\NN\setminus\{4\}$.
\end{itemize}
\end{theorem}

In none of the cases addressed in Theorem~\ref{CnoRU} the
eigenvalues are roots of unity, so Equation~\eqref{E:MPerPerOfEndo} 
implies $\MPer(f_A) = \Per(f_A)$.
\textit{These cases are listed in rows 12-15 of Table~2.}

\subsection{Real eigenvalues}

Real eigenvalues either come in pairs of irrational quadratic
integers (that is, zeros of polynomials of the shape $x^2 - tx +d$,
$t,d\in\ZZ$) or elements of $\ZZ$. The fact that rational
eigenvalues have to be integers is shown in the following lemma.

\begin{lemma}\label{L:rationalImpliesInt}
If $A$ has one (and thus two) rational eigenvalues, they are in fact
integers, hence the polynomial splits into two linear factors over
$\ZZ$.
\end{lemma}

\begin{proof}
Let $\frac{p_1}{q_1}$ and $\frac{p_2}{q_2}$ be the zeros of the
characteristic polynomial $\chi^{}_A(x)$. Without loss of
generality, we can assume that $\gcd(p_1,q_1) = \gcd(p_2,q_2)=1$. We
know that $\frac{p_1}{q_1} + \frac{p_2}{q_2} = \frac{p_1 q_2 + p_2
q_1}{q_1 q_2}= \tr(A) \in\ZZ$.

\smallskip

Consequently, the numerator $r$ has to be divisible by $q_1 q_2$,
hence it is divisible by both $q_1$ and $q_2$, or, equivalently, $r=
p_1 q_2 + p_2 q_1 \equiv 0 \mod q_i$, $i\in\{1,2\}$. But $r \equiv 0
\equiv p_1 q_2 \mod q_1$, whence $p_1 q_2$ is divisible by $q_1$. As
$\gcd(p_1,q_1)=1$ by assumption, one has $q_1 | q_2$. Considering
the equation modulo $q_2$ analogously gives $q_2 | q_1$, hence
$q\coloneqq q_1 =\pm q_2$. Furthermore, $\frac{p_1}{q_1} \cdot
\frac{p_2}{q_2} = \det(A)\in\ZZ$, thus $\pm q^2 | p_1 p_2$ such that
the assumed coprimality implies $q\in\{-1,1\}$.
\end{proof}

The following theorem is the summary of Propositions~4.18, 4.19 and
4.23 from \cite{ABLSS}.

\renewcommand{\labelenumi}{\textnormal{(\alph{enumi})}}
\begin{theorem}[\cite{ABLSS}]\label{T:realEV}
Let $A$ be in $\Mat(2,\ZZ)$ with real eigenvalues.
\begin{enumerate}
\item Suppose $A$ has eigenvalues, neither of which is $\pm1$.
If $d+t\notin\{0,-2\}$, then $\MPer(f_A)=\NN$.\label{realNotOne}

\item If $d=t=0$ then $\MPer(f_A) = \{1\}$.\label{nullnull}

\item Suppose $t+d\in\{0,-2\}$ and  $(t,d)\not= (0,0)$.
Then $\MPer(f_A) =\NN\setminus\{2\}$.
\end{enumerate}
\end{theorem}

In the situation of $(a)$ and $(b)$,
no root of unity is eigenvalue of the matrix $A$ and hence 
$\MPer(f_A) = \Per(f_A)$. 
\textit{The rows in Table~2 corresponding to Theorem~\ref{T:realEV} are 9, 16
and 17}.

\smallskip

So according to Equation \eqref{E:MPerPerOfEndo} and in view of
Theorem~\ref{T:realEV}, the only cases of real eigenvalues in which
$\MPer(f_A) \not= \Per(f_A)$ correspond to $\pm 1$ being among the
eigenvalues. As Lemma~\ref{L:rationalImpliesInt} tells us, the other
eigenvalue then has to be an integer as well. These cases will be
dealt with in the following section.

\section{Filling the gaps: the remaining cases}\label{S:completion}

The cases where the general theory of minimal period sets does not
immediately give the period set of the corresponding toral
endomorphism is the case when the eigenvalues are $1$ or $-1$ 
and some second integer, which is in fact decisive for the determination
of the period set.
The simple example of $\id$ and $\left(\begin{matrix}1&1\\0&1
\end{matrix}\right)$ shows that the same eigenvalues can a
priori imply different period sets, the former being $\{1\}$, the
latter $\NN$, although both matrices share the (double) eigenvalue
$1$.

\smallskip

Let $E = \{-2,-1,0,1\}$ and $\ord_n a$ the order of $a$ modulo $n$, that
is, the least integer $k>0$ such that $a^k\equiv 1 \mod n$.
The order $\ord_n (a)$ is well-defined for all $n\in\NN$, $n>1$, such that
$\gcd(a,n)=1$.

\begin{lemma}\label{L:ordersets}
Let $a\in\ZZ\setminus\{0\}$. Then
\[
P_a \coloneqq \{\ord_n a \mid n\in\NN\setminus\{1\}, \gcd(a,n)=1\}
\]
equals one of the following sets:
\begin{enumerate}
\item $\{1\}$ if $a=1$,

\item $\{1,2\}$ if $a=-1$,

\item $\NN\setminus\{2\}$ if $a=-2$,

\item $\NN\setminus\{1\}$ if $a=2$,

\item $\NN$ if $a\in\ZZ\setminus(\exSet\cup\{2\})$.
\end{enumerate}
\end{lemma}

\begin{proof}
The statements for $a=\pm 1$ are clear, so we exclude these cases
from our further considerations. We first check that $2\notin
P_{-2}$: if $(-2)^2 \equiv 1 \mod n$, clearly, $n| [(-2)^2 -1] = 3$,
so the only possible choice for $n$ is $3$ (as $1$ is excluded). But
$-2 \equiv 1 \mod 3$, and thus $\ord_{3} (-2) = 1$, whence $1\in
P_{-2}$ and $2\notin P_{-2}$ follows.

\smallskip

Next, we note that $1\notin P_{2}$, for $2\equiv 1 \mod n$ would
imply $n|(2-1)=1$; however, $1\in P_a$ for all other $a$ because $a
\equiv 1 \mod (a-1)$ if $a>2$ and $a \equiv 1 \mod (-a+1)$ if $a <
-1$.

\smallskip

We first consider the case where $a$ is positive and assume $a>2$
or $k>1$, such that $a^k-1 > 1$. Clearly, $a^k \equiv 1 \mod
(a^k-1)$, so assume there is a $d|k$, $d\not= k$, with $a^d \equiv 1
\mod (a^k-1)$. But then $(a^k - 1) | (a^d -1)$, which is not
possible for $d<k$, hence $\ord_{a^k-1} (a) =k$ and $k\in P_a$.

\smallskip

Next, consider the case where $a$ is negative and $k$ even, $a\not=
-2$ or $k\not= 2$. One has $a^k \equiv 1 \mod (a^k-1)$, and $(a^k -
1)| (a^d - 1)$ for the divisor $d = \ord_{a^k-1}(a)$ of $k$. If $d$
is even, this gives the same contradiction as above; if $d$ is odd,
one has $(a^k-1) | (-a^d+1)$ and thus $\abs{a}^k - \abs{a}^d \leq
2$. But the only positive integers $d,k$ with $d<k$ satisfying this
inequality are $k=2$, $d=1$ in the case $\abs{a} = 2$. For $a < -2$
and $k$ odd, one readily checks that $a^k \equiv 1 \mod (-a^k+1)$
and similar reasoning as above shows that $k$ is indeed the order
modulo $(-a^k+1)$. Hence it follows that if $\abs{a} > 2$ or $k>2$
then $k\in P_a$, and in summary, the lemma follows.
\end{proof}

We stipulate that $P^{}_0 = \emptyset$. Then we can formulate the
following

\begin{prop}\label{P:intEV}
Let $A$ be in $\Mat(2,\ZZ)$ and assume $A$ has the eigenvalues $a, b
\in \ZZ$. Then $\Per(f_A) \supset \{1\}\cup P_{a}\cup P_{b}$.
\end{prop}

\begin{proof}
According to Theorem~\ref{T:mgcd}, two matrices in $\Mat(2,\ZZ)$
induce the same orbit structure on the rational lattices $L_n$ for
all $n$, if trace, determinant and mgcd coincide. The square of the
mgcd always divides the discriminant, see Equation
\eqref{E:mgcdDivDiscr}, hence the set of possible triplets
$(\tr,\det,\mgcd)$ can be itemised explicitly. Fixing the trace and
the determinant as $\tr(A) = a+b, \det(A) = a\cdot b$, the
discriminant is $\Delta = \tr(A)^2 - 4\det(A) = (b-a)^2$, whence the
mgcd has to divide $b-a$. Consequently, one has an upper diagonal
matrix in the following set corresponding to each possible value of
the mgcd
\begin{equation}\label{E:normalForms}
\left\{\left(\begin{matrix}a & r\\0 & b\end{matrix}\right) \mid r^2
| \Delta \right\}.
\end{equation}
In other words, given an integer matrix $A$, there is some upper
diagonal matrix in the above set that has exactly the same orbit
structure on each of the finite sets $L_n$, $n\in\NN$. Furthermore,
the matrices $\left(\begin{matrix}a & r\\ 0 & b\end{matrix}\right)$
and $\left(\begin{matrix}b & r\\ 0 & a\end{matrix}\right)$ have the
same orbit structure on each $L_n$. But an upper diagonal matrix
$\left(\begin{matrix}d_1 & r\\0 & d_2\end{matrix}\right)$ clearly
has points of all orders contained in $P_{d_1}$ and $P_{d_2}$, as,
for instance, for $k\in P_{d_1}$, the point $(\frac{1}{n},0)$ has
period $k=\ord_n(d_1)$ and by the previous sentence, the roles of
$d_1$ and $d_2$ are interchangeable. Thus, for eigenvalues $a,b$,
the set $\Per(f_A)$ comprises the union $P_a \cup P_b$. Since the
point $(0,0)$ has period $1$ for every integer matrix, the claim
follows.
\end{proof}

\begin{remark}
Note that the eigenvalue $0$ does not contribute any periods. If
both eigenvalues are $0$, the matrix is nilpotent, and the only
periodic point is $(0,0)$. If the second eigenvalue is non-zero, it
has to be an integer, and it determines the period set according to
Lemma~\ref{L:ordersets} and Proposition~\ref{P:intEV}.
If the second eigenvalue is $1$ or $-1$, one has
$A^2=A$ or $A^3=A$, respectively, which immediately gives the period sets.
\textit{The cases with $0$ among the eigenvalues are treated in rows
9--11 of Table~2 if not covered by row 16 or 17.}

\smallskip

The symbol `$\supset$' in Proposition~\ref{P:intEV} cannot be replaced by an
equality sign, because a non-zero upper-diagonal entry $r$ may give
rise to additional periods, as in the aforementioned example of
$\id$ versus $\left(\begin{matrix}1 & 1\\ 0 & 1\end{matrix}\right)$,
as well as for the matrices
$-\id$ and $\left(\begin{matrix}-1 & \phantom{-}1\\ \phantom{-}0 & -1\end{matrix}\right)$.
These cases are covered by Proposition~3.6 in \cite{TP}, where the powers of 
suitably parametrised $2\times2$ matrices are considered.
Using our local approach modulo $n$, one could restrict the considerations to upper
diagonal matrices as in Equation~\eqref{E:normalForms}, which, together with the Nielsen numbers
whenever they do not vanish, yield the period sets as listed in Table~2.
We note that the cases of a double eigenvalue $1$ or $-1$, respectively, are the
only ones in which the knowledge of the eigenvalues is not enough to determine 
the period set of the endomorphism.
\textit{The cases corresponding to the eigenvalues $\pm 1$ are listed in rows 1--5 in Table~2.}
\end{remark}

With help of Proposition~\ref{P:intEV}, we can write down the period
set for all remaining cases where $1$ or $-1$ is among the
eigenvalues of the matrix.

\begin{coro}\label{C:evpm1}
   Let $(t,d)$ denote the pair of trace and determinant of
   $A\in\Mat(2,\ZZ)$.
   \begin{itemize}
   \item[(a)] If the eigenvalues are $-1, -d$, hence $t+d=-1$, $d\notin\{-1,0,1\}$
   the period set $\Per(f_A)$ is $\NN$.

   \item[(b)] If the eigenvalues are $-2, 1$, hence $t-d=1$, then
   $\Per(f_A) = \NN\setminus\{2\}$.

   \item[(c)] If the eigenvalues are $1,d$, hence $t-d = 1$, 
   $(t,d)\not=(-1,-2)$,
   then $\Per(f_A) = \NN$.
   \end{itemize}
\end{coro}
\begin{proof}
   Case (a) and (c) immediately follow from Proposition~\ref{P:intEV};
   for (b), we note that for each $x$ with $A^2 x = x \mod 1$, one finds
   $A^2 x - \one = -A x + x = -(A-\id) x = 0 \mod 1$ by the Cayley-Hamilton 
   Theorem,  so $x$ is a fixed point
   and therefore, $A$ does not admit any points of least period $2$.
\end{proof}

We state again that the cases treated in Corollary~\ref{C:evpm1} are
the ones in which $\Per(f_A)$ and $\MPer(f_A)$ significantly differ;
the latter can be found on page~30 in \cite{ABLSS}.
\\
\textit{Corollary~\ref{C:evpm1} yields the entries of rows 18--20 in
Table~2.}

\section*{Acknowledgements}

The first author is partially supported by a MCYT/FEDER grant
MTM2008--03437, an AGAUR grant number 2014SGR--568, an ICREA
Academia, two grants FP7-PEOPLE-2012-IRSES 316338 and 318999, and
FEDER-UNAB10-4E-378. The second author 
gratefully acknowledges the funding of the four months stay at the 
Universitat Aut\`{o}noma de Barcelona
within the project Rozvoj v\v{e}deck\'ych kapacit Slezsk\'e univerzity
v Opav\v{e} (CZ.1.07/2.3.00/30.0007).


\begin{thebibliography}{0}

\bibitem{ALM}
L. Alsed\`{a}, J. Llibre and M. Misiurewicz, {\it Combinatorial dynamics
and entropy in dimension one}, second edition, Advanced Series in
Nonlinear Dynamics {bf 5}, World Scientific Publishing Co., Inc.,
River Edge, NJ, 2000.

\bibitem{ALM2}
L. Alsed\`{a}, J. Llibre and M. Misiurewicz, {\it Low--dimensional
combinatorial dynamics. Discrete dynamical systems}, Internat. J.
Bifur. Chaos Appl. Sci. Engrg. {\bf 9} (1999), 1687--1704.

   \bibitem{ABLSS}
   L.~Alsed\'a, S.~Baldwin, J.~Llibre, R.~Swanson and
   W.~Szlenk,
   Minimal sets of periods for torus maps via Nielsen numbers,
   \textit{Pacific J. Math.} \textbf{169} (1995) 1--32.

   \bibitem{B}
   L.S.~Block, Periodic orbits of continuous mappings of the circle,
   \textit{Trans. Amer. Math. Soc.} \textbf{260} (1980) 555--562

   \bibitem{BC}
   L.S.~Block and W.A.~Coppel,
   \textit{Dynamics in one dimension},
   Lecture Notes in Math. {\bf 1513}, Springer, Berlin, 1992.

\bibitem{BGMY}
L. Block, J. Guckenheimer, M. Misiurewicz and L.S. Young,
\textit{Periodic points and topological entropy of one-dimensional
maps}, Lecture Notes in Math. {\bf 819}, Springer, Berlin, 1980, pp.
18--34.

   \bibitem{BRW}
   M.~Baake, J.A.G.~Roberts and A.~Weiss,
   Periodic orbits of linear endomorphisms of the $2$-torus
   and its lattices.
   \textit{Nonlinearity} \textbf{21} (2008) 2427--2446.

   \bibitem{BHP}
   M.~Baake, J.~Hermisson and P.A.B.~Pleasants,
   The torus parametrization of quasiperiodic LI-classes.
   \textit{J.\ Phys.~A:\ Math.\ Gen.}  \textbf{30} (1997) 3029--3056.

   \bibitem{BNR}
   M.~Baake, N.~Neum\"arker and J.A.G.~Roberts,
   Orbit structure and (reversing) symmetries
   of toral endomorphisms on rational lattices,
   \textit{Discrete Contin. Dyn. Syst.} \textbf{33} (2013) 527--553.

   \bibitem{BF}
   E.~Behrends and B.~Fiedler,
   Periods of discretized linear Anosov maps.
   \textit{Ergod.\ Th.\ \& Dynam.\ Syst.} \textbf{18} (1998) 331--341.

   \bibitem{H}
   B.~Halpern,
   Periodic points on tori,
   \textit{Pacific J. of Math.} \textbf{83} (1979), 117--133.

\bibitem{JM}
J. Jezierski and W. Marzantowicz,{\it Homotopy methods in
topological fixed and periodic points theory}, Topological Fixed
Point Theory and Its Applications {\bf 3}, Springer, Dordrecht,
2006.

\bibitem{JL}
B. Jiang and J. Llibre, {\it Minimal sets of periods for torus
maps}, Discrete Contin. Dynam. Systems {\bf 4} (1998), 301--320.

   \bibitem{K}
   V.~Kannan,
   Sets of periods of dynamical systems,
   \textit{Indian J. Pure Appl. Math.} \textbf{41}(1) (2010) 225--240.

   \bibitem{TP}
   V.~Kannan, I.~Subramania Pillai, K. Ali Akbar and B.~Sankararao,
   The set of periods of periodic points of a toral endomorphism,
   \textit{Topology Proceedings} \textbf{37} (2011) 1--14.

   \bibitem{KH}
   A.~Katok and B.~Hasselblatt,
   \textit{Introduction to the Modern Theory of Dynamical Systems},
   Cambridge University Press, Cambridge, 1995.

   \bibitem{KP}
   J.~Kuzmanovich and A.~Pavlichenkov,
   Finite groups of matrices whose entries are integers,
   \textit{Amer. Math. Monthly} \textbf{109} (2002) 173--186.

   \bibitem{Ll}
   J. Llibre, {\it A note on the set of periods for Klein bottle maps},
   Pacific J. Math. {\bf 157} (1993), 87--93.

   \bibitem{Mi}
   M. Misiurewicz, {\it Periodic points of maps of degree one of a
   circle}, Ergodic Theory Dynamical Systems {\bf 2} (1982), 221--227.

   \bibitem{PV}
   I.~Percival and F.~Vivaldi,
   Arithmetical properties of strongly chaotic motions,
   \textit{Physica} \textbf{25\ts D} (1987) 105--130.

   \bibitem{W}
   P.~Walters,
   \textit{An Introduction to Ergodic Theory},
   reprint, Springer, New York, 2000.
\end{thebibliography}
\end{document}